\documentclass[aap,seceqn,MSNbibl,dvips]{arximspdf}

\doi{10.1214/10-AAP693}
\volume{21}
\issue{1}
\pubyear{2011}
\firstpage{240}
\lastpage{265}

\makeatletter
\newtheorem{theorem}{Theorem}[section]
\newtheorem{prop}{Proposition}[section]
\newtheorem{cor}{Corollary}[section]
\newtheorem{lem}{Lemma}[section]

\newcommand{\eqref}[1]{(\ref{#1})}

\newcommand{\EE}{{\mathbb E}}
\newcommand{\sC}{{\mathcal C}}
\newcommand{\s}{\sigma}
\newcommand{\qq}{\qquad}
\newcommand{\q}{\quad}
\newcommand{\bbb}{\beta}
\newcommand{\aalpha}{\alpha}

\newcommand{\ggamma}{\gamma}
\newcommand{\ttheta}{\theta}
\newcommand{\rc}{random-cluster}
\newcommand{\NN}{{\mathbb N}}
\newcommand{\ZZ}{{\mathbb Z}}
\newcommand{\RR}{{\mathbb R}}
\newcommand{\ra}{\rangle}
\newcommand{\Si}{\Sigma}
\newcommand{\Om}{\Omega}
\newcommand{\De}{\Delta}
\newcommand{\La}{\Lambda}
\newcommand{\oo}{\infty}
\newcommand{\lest}{\le_{\mathrm{st}}}
\newcommand{\pc}{p_{\mathrm c}}
\newcommand{\bc}{\beta_{\mathrm c}}
\newcommand{\hc}{h_{\mathrm c}}
\newcommand{\psd}{p_{\mathrm{sd}}}
\newcommand{\Tc}{{T_{\mathrm c}}}
\newcommand{\rad}{{\mathrm{rad}}}
\newcommand{\pibh}{\pi_{\bbb,h}}
\newcommand{\piLbh}{\pi_{\La,\bbb,h}}
\newcommand{\pibho}{\pi_{\bbb,h_1}}
\newcommand{\pibht}{\pi_{\bbb,h_2}}
\newcommand{\pibhc}{\pi_{\bbb,\hc}}
\newcommand{\pibz}{\pi_{\bbb,0}^+}
\newcommand{\fpqo}{\phi_{p_1,q}}
\newcommand{\fpqt}{\phi_{p_2,q}}
\newcommand{\pin}{\pi_{n,h}}
\newcommand{\pinhcdp}{\pi_{n,\hc+{\delta }}}
\newcommand{\pinhcdm}{\pi_{n,\hc-{\delta }}}
\newcommand{\ol}[1]{\overline{#1}}
\newcommand{\gest}{\ge_{\mathrm{st}}}
\newcommand{\sm}{\setminus}
\newcommand{\resp}{resp.}
\newcommand{\ci}{\cite}
\newcommand{\rd}{\mathrm{d}}
\newcommand{\Gd}{G_\rd}
\newcommand{\prd}{p_\rd}
\newcommand{\kq}{\psd(q)} 
\newcommand{\hkm}{H_{k,m}}
\newcommand{\bkm}{B_{k,m}}
\newcommand{\jkmn}{J_{k,m,n}}
\newcommand{\pd}{\partial}
\newcommand{\lra}{\leftrightarrow}
\newcommand{\es}{\varnothing}
\newcommand{\fpq}{\phi_{p,q}}
\newcommand{\fpdq}{\phi_{\prd,q}}
\newcommand{\fnp}{\phi_{n,p}}
\newcommand{\fnpd}{\phi_{n,p_\rd,\rd}}
\newcommand{\fnpq}{\phi_{\La_n,p,q}}
\newcommand{\kpqa}{\kappa_{p,q,\aalpha}}
\newcommand{\knh}{\kappa_{n,h}}
\newcommand{\ppqa}{\pi_{p,q,\aalpha}}
\newcommand{\ppqah}{\pi_{p,q,\aalpha,h}}
\newcommand{\psh}{\phi_h}
\newcommand{\pnsh}{\phi_{n,h}}
\newcommand{\pnh}{\pi_{n,h}}
\newcommand{\ph}{\pi_h}
\newcommand{\pshV}{\phi_{V,h}}
\newcommand{\pspqahh}{\phi_{p,q,\aalpha,h}}
\newcommand{\phV}{\pi_{V,h}}\def\psh{\phi_h}

\newcommand{\kh}{\kappa_h}
\newcommand{\kpqah}{\kappa_{p,q,\aalpha,h}}
\newcommand\lrao[1]{\stackrel{#1}{\leftrightarrow}} 
\newcommand\lrap{\stackrel{+}{\leftrightarrow}} 
\newcommand\lram{\stackrel{-}{\leftrightarrow}_*} 
\newcommand\lrab{\stackrel{1}{\leftrightarrow}} 
\newcommand\lraw{\stackrel{0}{\leftrightarrow}_*} 

\makeatother

\begin{document}
\begin{frontmatter}

\title{Sharp thresholds for the random-cluster \break and Ising models}
\runtitle{Random-cluster and Ising models}

\begin{aug}
\author[A]{\fnms{Benjamin}
\snm{Graham}\ead[label=e1]{graham@dma.ens.fr}\ead[label=u1,url]{http://www.dma.ens.fr/\texttildelow graham/}}
\and
\author[B]{\fnms{Geoffrey} \snm{Grimmett}\ead
[label=e2]{g.r.grimmett@statslab.cam.ac.uk}\ead[label=u2,url]{http://www.statslab.cam.ac.uk/\texttildelow grg/}\corref{}}
\runauthor{B. Graham and G. Grimmett}
\affiliation{University of British Columbia and Cambridge University}
\address[A]{Department of Mathematics\\
University of British Columbia\\
Vancouver, B. C.\\
Canada V6T 1Z2\\
and\\
DMA\\
\'Ecole Normale Sup\'erieure\\
45 rue d'Ulm\\
75230 Paris Cedex 5\\
France\\
\printead{e1}\\
\printead{u1}}
\address[B]{Statistical Laboratory\\
Centre for Mathematical Sciences\\
Cambridge University\\
Wilberforce Road\\
Cambridge CB3 0WB\\
UK\\
\printead{e2}\\
\printead{u2}}
\end{aug}

\received{\smonth{3} \syear{2009}}
\revised{\smonth{2} \syear{2010}}

%
\begin{abstract}
A sharp-threshold theorem is proved for box-crossing
probabilities on the square lattice. The models in question
are the \rc\ model near the self-dual point
$\psd(q)=\sqrt q/(1+\sqrt q)$, the Ising model with
external field, and the colored \rc\ model.
The principal technique is an extension
of the influence theorem for monotonic probability measures
applied to increasing events with no assumption of symmetry.
\end{abstract}


\end{frontmatter}

\section{Introduction}\label{sec:intro}
The method of ``sharp threshold'' has been fruitful in
probabilistic combinatorics (see \cite{G-pgs,KS05} for recent reviews).
It provides a fairly robust tool
for showing the existence of a sharp threshold for certain
processes governed by independent random variables. Its most compelling
demonstration so far in the field of physical systems
has been the proof in \cite{BR2} that the critical probability of site
percolation
on the Voronoi tessellation generated by a Poisson process on $\RR^2$
equals $\frac12$.

Each of the applications alluded to above involves a product measure.
It was shown in
\cite{GG} that the method
may be extended to nonproduct probability measures satisfying the FKG
lattice condition.
The target of this note is to present two applications of such a
sharp-threshold theorem to measures arising in statistical physics,
namely those of the \rc\ model and the Ising model. In each case, the
event in question is the
existence of a crossing of a large box, by an open path in the case of
the \rc\ model,
and by a single-spin path in the case of the Ising model. A related but
more tentative
and less complete result has
been obtained in \cite{GG} in the first case, and the second case has been
studied already in \cite{vdB07} and \cite{Hig0,Higuchi}.

Our methods for the Ising model can be applied to a more general
model termed here the colored \rc\ model (CRCM), see Section \ref{sec:crcm}.
This model is related
to the so-called fractional Potts model of \cite{KahnW}, and the
fuzzy Potts model and the divide-and-color model of \cite
{DaCmodel,Chay96,Hag99,Hag01}.

The sharp-threshold theorem used here is an
extension of that given for product measure in
\cite{FKST,Tal94}, and it makes use of the results of \cite{GG}.
It is stated, with an outline of the proof, in Section \ref{sec:ist}.
The distinction of the current sharp-threshold theorem is that it makes
no assumption of symmetry on either the event or measure in question. Instead,
one needs to estimate the maximum influence of the various components,
and it turns out that this may be done in a manner which is very idiomatic
for the models in question.
The sharp-threshold theorem presented here may find further applications
in the study of dependent random variables.

\section{The models}\label{sec:models}

\subsection{The \rc\ model}
The \rc\ model on a connected
graph $G$ has two parameters: an edge-weight $p$ and a cluster-weight
$q$. See Section~\ref{sec:bc} for a formal definition.
When $q\ge1$ and $G$ is infinite, there is a critical value $\pc(q)$
that separates the subcritical phase of the model
[when $p<\pc(q)$ and there exist no infinite
clusters] and the supercritical phase. It has long been \vspace{1pt} conjectured that,
when $G$ is the square lattice $\ZZ^2$,
%
\begin{equation}\label{pcq}
\pc(q) = \frac{\sqrt q}{1+\sqrt q}, \qq q\ge1.
\end{equation}
This has been proved rigorously in three
famous cases. When $q=1$, the \rc\ model
is bond percolation, and the exact calculation $\pc(1)=\frac12$ was
shown by Kesten \ci{Ke80}. When $q=2$, the model is intimately related
to the Ising
model, and the calculation of $\pc(2)$ is equivalent to that of
Onsager and others
concerning the Ising critical temperature (see \cite{ABF,AF} for a
modern treatment
of the Ising model). Formula \eqref{pcq}
has been proved for sufficiently large values of
$q$ (currently $q \ge21.61$) in the context of
the proof of first-order phase transition, see \ci{G-RC,Kot-S,LMMRS,LMR}.
We recall that, when $q\in\{2,3,\dots\}$, the critical temperature
$\Tc$ of the $q$-state
Potts model on a graph $G$ satisfies
%
\begin{equation}\label{Pottscp}
\pc(q)= 1-e^{-1/\Tc}.
\end{equation}
A fairly full account of the \rc\ model, and its relation to the Potts model,
may be found in \cite{G-RC}.

Conjecture \eqref{pcq} is widely accepted. Physicists have proceeded
beyond a ``mere'' calculation of the critical point, and have
explored the behavior of the process at and near this value.
For example, it is believed that there is a continuous (second-order)
phase transition
if $1\le q < 4$, and a discontinuous (first-order) transition when $q > 4$,
see \cite{Bax}. Amongst recent progress, we highlight the
stochastic L\"owner evolution process SLE$_{16/3}$ associated with
the cluster boundaries in the critical case when $q=2$ and $p=\sqrt
2/(1+\sqrt2)$,
see \cite{Smir,Smi07}.

The expression in \eqref{pcq} arises as follows through the use of
planar duality.
When the underlying graph $G$ is planar, it possesses a (Whitney) dual graph
$\Gd$. The \rc\ model on $G$ with parameters $p$, $q$ may be related to
a dual \rc\ model on $\Gd$ with parameters $\prd$, $q$, where
%
\begin{equation}\label{dualv}
\frac{\prd}{1-\prd} = \frac{q(1-p)}p.
\end{equation}
The mapping $p \mapsto\prd$ has a fixed point $p=\kq$, where
\[
\kq:=\frac{\sqrt q}{1+\sqrt q}
\]
is termed the \textit{self-dual point}. The value $p=\kq$ is
especially interesting when $G$ and $\Gd$ are isomorphic, as in the case
of the square lattice $\ZZ^2$. See \cite{G-RC}, Chapter~6.
We note for future use that
%
\begin{equation}\label{mel60}
p < \kq \q \mbox{if and only if} \q \prd> \kq.
\end{equation}

Henceforth, we take $G=\ZZ^2$.
The inequality
%
\begin{equation}\label{pcqge}
\pc(q) \ge \kq, \qq q \ge1,
\end{equation}
was proved in \cite{G95,Wel93} using Zhang's argument (see \cite{G99}, page 289).
Two further steps would be enough to imply the complementary
inequality $\pc(q) \le\kq$: firstly, that the probability of crossing
a box $[-m,m]^2$ approaches 1 as $m\to\oo$, when $p>\kq$; and
secondly, that
this implies the existence of an infinite cluster. The first of these
two claims
is proved in Theorem \ref{thm1}.

Kesten's proof for percolation, \cite{Ke80}, may be viewed as a proof
of the
first claim in the special case $q=1$. The second claim follows
for percolation by RSW-type arguments, see \cite{Ru78,Ru81,SeW} and
\cite{G99}, Section 11.7.
Heavy use is made in these works of the fact
that the percolation measure is a product measure, and this is where
the difficulty
lies for the \rc\ measure.

We prove our main theorem (Theorem \ref{thm1} below) by the method of
influence and sharp threshold
developed for product measures in \cite{FKST,KKL}. This was adapted in
\cite{GG}
to monotonic measures applied to increasing events, subject to a certain
hypothesis of symmetry. We show in Section \ref{sec:ist} how this
hypothesis may be
removed, and we apply the subsequent inequality in Section \ref{sec:pf} to the probability of
a box-crossing, thereby extending to general $q$ the corresponding
argument of \cite{BR1}.

\subsection{Ising model}
We shall consider the Ising model on the square lattice $\ZZ^2$ with
edge-interaction
parameter $\bbb$ and external field $h$. See Section \ref{sec:ising}
for the relevant definitions.
Write $\bc$ for the critical value of $\bbb$ when $h=0$, so that
\[
1-e^{-2\bc} = \psd(2),
\]
where $\psd(2)$ is given as in \eqref{pcq}. Two notions of
connectivity are required:
the usual connectivity relation $\lra$ on $\ZZ^2$ viewed as a graph,
and the relation
$\lra_*$, termed $*$-connectivity, and
obtained by adding diagonals to each unit face of $\ZZ^2$. Let $\pibh
$ denote the
Ising measure on $\ZZ^2$ with parameters $\bbb$, $h$.

Higuchi proved in \cite{Hig0,Higuchi} that, when $\bbb\in(0,\bc)$,
there exists
a critical value $\hc=\hc(\bbb)$ of the external field such that:
\begin{longlist}[(b)]
\item[(a)]$\hc(\bbb)>0$,
\item[(b)] when $h >\hc$, there exists $\pibh$-almost-surely an infinite
$+$ cluster of $\ZZ^2$, and the radius of the $*$-connected $-$
cluster at the origin
has exponential tail,
\item[(c)] when $0<h<\hc$, there exists $\pibh$-almost-surely an infinite
$*$-connected $-$ cluster of $\ZZ^2$, and the radius of the $+$
cluster at the origin
has exponential tail.
\end{longlist}
A further approach to Higuchi's theorem has been given recently by van
den Berg \cite{vdB07}.
A key technique of the last paper is a sharp-threshold theorem of
Talagrand \cite{Tal94}
for product measures. The Ising measure $\piLbh$ on a box $\La$
is of course not a product measure, and
so it was necessary to encode it in terms of a family of independent random
variables. We show here that the influence theorem of \cite{GG} may be
extended and applied
directly to the Ising model to obtain the necessary sharp threshold result.
(The paper \cite{vdB07} contains results for certain other models
encodable in terms of product measures, and these appear to be beyond
the scope
of the current method.)

\subsection{Colored \rc\ model}
The Ising model with external field is a special case of a class
of systems that have been studied by a number of authors, and which
we term \textit{colored \rc\ models} (CRCM). Sharp-threshold results
may be obtained
for such systems also. Readers are referred to Section \ref{sec:crcm}
for an account of the CRCM and the associated results.

\section{Box-crossings in the \rc\ model}\label{sec:bc}
The \rc\ measure is given as follows on a finite graph $G=(V,E)$.
The configuration space is $\Om=\{0,1\}^E$. For $\omega \in\Om$, we write
$\eta(\omega )=\{ e\in E\dvtx\omega
(e)=1\}$ for the set of ``open'' edges, and $k(\omega )$ for the number of
connected components in the open graph $(V,\eta(\omega ))$. Let
$p\in[0,1]$, $q\in(0,\oo)$, and let $\fpq$ be the probability
measure on $\Om$ given by
%
\begin{equation}\label{rcmeas}
\fpq(\omega )=\frac{1}{Z}\,\biggl\{\prod_{e\in E} p^{\omega (e)}
(1-p)^{1-\omega (e)}\biggr\} q^{k(\omega )} ,\qq\omega \in\Om,
\end{equation}
where $Z=Z_{G,p,q}$ is the normalizing constant.
We shall assume throughout this paper that $q\ge1$,
so that $\fpq$ satisfies the so-called FKG lattice condition
%
\begin{equation}\label{4.2}
\mu(\omega _1\vee\omega _2)\mu(\omega _1\wedge\omega _2)\geq
\mu(\omega _1)\mu
(\omega _2), \qq\omega _1,\omega _2\in\Om.
\end{equation}
Here, as usual,
\begin{eqnarray*}
\omega _1\vee\omega _2(e)&=&\max\{\omega _1(e),\omega _2(e)\},\\
\omega _1\wedge\omega _2(e)&=&\min\{\omega _1(e),\omega _2(e)\}
\end{eqnarray*}
for $e \in E$. As a consequence of \eqref{4.2}, $\fpq$ satisfies the
FKG inequality.
See \cite{G-RC} for the basic properties of the \rc\ model.

Consider the square lattice $\ZZ^2$ with edge-set $\EE$,
and let $\Om=\{0,1\}^\EE$.
Let $\La=\La_n=[-n,n]^2$ be a finite box of $\ZZ^2$, with edge-set
$\EE_\La$. For $b\in\{ 0,1\}$ define
\[
\Om_\La^b=\{\omega \in\Om\dvtx\omega (e)=b\;\mbox{for}\; e\notin\EE
_\La\}.
\]
On $\Om_\La^b$ we define a random-cluster measure
$\phi_{\La,p,q}^b$ as follows. For $p\in[0,1]$ and $q\in[1,\oo)$, let
%
\begin{equation}\label{13.6}
\qquad\phi_{\La,p,q}^b(\omega )=\frac{1}{Z_{\La
,p,q}^b}\,\biggl\{\prod_{e\in\EE_\La} p^{\omega (e)}
(1-p)^{1-\omega
(e)}\biggr\} q^{k(\omega ,\La)},\qq\omega \in\Om_\La^b,
\end{equation}
where $k(\omega ,\La)$ is the number of clusters of $(\ZZ^2,\eta
(\omega ))$
that intersect $\La$. The boundary condition $b=0$ (\resp,
$b=1$) is usually termed ``free'' (\resp, ``wired'').
It is standard that the weak limits
\[
\fpq^b = \lim_{n\to\oo} \phi_{\La_n,p,q}^b
\]
exist, and that they are translation-invariant, ergodic, and
satisfy the FKG inequality. See \cite{G-RC}, Chapter 4.

For $A,B \subseteq\ZZ^2$, we write $A \lra B$ if there exists an open path
joining some $a \in A$ to some $b \in B$. We write $x \lra\oo$ if the
vertex $x$
is the endpoint of some infinite open path. The percolation
probabilities are given as
\[
\ttheta^b(p,q) = \fpq^b(0 \lra\oo), \qq b=0,1.
\]
Since each $\ttheta^b$ is nondecreasing in $p$, one may define the critical
point by
\[
\pc(q) = \operatorname{sup}\{p\dvtx \ttheta^1(p,q)=0\}.
\]
It is known that $\fpq^0 = \fpq^1$ if $p \ne\kq$, and we write
$\fpq$ for
the common value. In particular,
$\ttheta^0(p,q) = \ttheta^1(p,q)$ for $p \ne\pc(q)$.
It is conjectured that
$\fpq^0 = \fpq^1$ when $p=\pc(q)$ and $q \le4$.

Let $B_k=[0,k]\times[0,k-1]$, and let $H_k$ be the event
that $B_k$ possesses an open left--right crossing.
That is, $H_k$ is the event that $B_k$ contains an open path having one
endvertex on its
left side and one on its right-hand side.

\begin{theorem}\label{thm1}
Let $q \ge1$. We have that
%
\begin{eqnarray}
\fpq(H_k) &\le& 2\rho_k^{\psd-p}, \qquad\hspace*{18pt} 0<p<\kq,\label{new10}\\
\fpq(H_k) &\ge& 1 - 2\nu_k^{p-\psd},\qquad \kq< p < 1\label{new11},
\end{eqnarray}
for $k \ge1$, where
%
\begin{equation}\label{mel26}
\rho_k = [2q\eta_{k}/p]^{c/q}, \qquad
\nu_k = [2q\eta_k/p_{\rd}]^{c/q}
\end{equation}
and
%
\begin{equation}\label{mel50}
\eta_k = \phi_{\kq,q}^0(0 \lra\pd\La_{k/2}) \to 0\qquad\mbox{as }k \to\oo.
\end{equation}
Here, $c$ is an absolute
positive constant, and $p_{\rd}$ satisfies \eqref{dualv}.
\end{theorem}

When $k$ is odd, we interpret $\pd\La_{k/2}$ in \eqref{mel50} as
$\pd\La_{\lfloor k/2\rfloor}$.

In essence, the probability of a square-crossing has
a sharp threshold around the self-dual ``pivot'' $\kq$. Related results were
proved in \cite{GG}, but with three relative weaknesses, namely: only
nonsquare
rectangles could be handled, the ``pivot'' of the threshold theorems was
unidentified,
and there was no result for \textit{infinite-volume} measures. The above
strengthening
is obtained by using the threshold Theorem \ref{inf1} which \vspace{1pt} makes no
assumption of symmetry
on the event or measure in question. The corresponding threshold
theorem for product
measure leads to a simplification of the arguments of \cite{BR1}
for percolation, see \cite{G-pgs}, Section 5.8.

Since $\fnpq^0 \lest\fpq\lest\fnpq^1$ and $H_k$ is an increasing
event, Theorem \ref{thm1} implies
certain inequalities for finite-volume probabilities also.

No estimate for the rate at which $\eta_k\to0$
is implicit in the arguments of this paper, and indeed one of the
targets of the current
work is to show that no estimate is necessary for sharp threshold.
It is expected that $\eta_k\to0$ at a rate that depends on whether or not
the phase transition is continuous: one expects
that $\eta_k$ decays as a power when
$1\le q<4$, and as an exponential when $q>4$ (see \cite{G-RC}, Section 6.4).
This would imply
a threshold of order either $1/\log k$ or $1/k$ in
\eqref{new10}--\eqref{new11}.
That the radius $R$ of the open cluster at the origin is $\phi_{\kq
,q}^0$-a.s. finite is a
consequence of the (a.s.) uniqueness of the infinite open cluster whenever
it exists. See \cite{G-RC}, Theorem 6.17(a), for a proof of the relevant
fact that
%
\begin{equation}\label{theta=0}
\ttheta^0(\kq,q) = 0, \qq q \ge1.
\end{equation}

We shall prove a slightly more general result than Theorem \ref{thm1}.
Let $B_{k,m} = [0,k]\times[0,m]$ and let $\hkm$ be the event that
there exists an open left--right crossing of $B_{k,m}$.

\begin{theorem}\label{thm3}
Let $q \ge1$.
We have that
%
\begin{eqnarray}
\qquad\fpqo(\hkm)[1-\fpqt(\hkm)] &\le&\rho_k^{p_2-p_1}, \qq
0<p_1<p_2\le\kq,\label{mel25}\\
\qquad\fpqo(\hkm)[1-\fpqt(\hkm)] &\le&\nu_{m+1}^{p_2-p_1}, \qq\kq\le
p_1<p_2<1,\label{mel25b}
\end{eqnarray}
for $k,m\ge1$, where $\rho_k$ (\resp, $\nu_k$) is given in \eqref{mel26}
with $p = p_1$ (\resp, $p=p_2$), and $\phi_{\kq,q}$
is to be interpreted as $\phi_{\kq,q}^0$.
\end{theorem}

\section{Box-crossings in the Ising model}\label{sec:ising}
Let $\La$ be a box of $\ZZ^2$. The spin-space of the Ising model on
$\La$
is $\Si_\La= \{-1,+1\}^{\La}$, and the Hamiltonian is
\[
H_\La(\s) = -\beta\sum_{e=\langle x,y\rangle\in\EE_\La} \s_x\s
_y - h \sum_{x\in\La} \s_x,
\]
where $\bbb> 0$, $h \ge0$. The relevant Ising measure is given by
\[
\piLbh(\s) \propto e^{-H_\La(\s)}, \qq\s\in\Si_\La,
\]
and it is standard that the (weak)
limit measure $\pibh= \lim_{\La\to\ZZ^2} \piLbh$ exists.
We shall also need the $+$ boundary-condition measure $\pibz$ given as
the weak limit
of $\pi_{\La,\bbb,0}$ conditional on $\s_x=+1$ for $x \in\pd\La$.
(Here, $\pd\La$ denotes as usual the boundary of $\La$, that is, the
set of
$x \in\La$ possessing a neighbor not belonging to $\La$).
By the FKG inequality or otherwise, $\pibz(\s_0)\ge0$, and the
critical value of $\bbb$ when $h=0$ is
given by
\[
\bc= \operatorname{sup}\{\bbb\dvtx \pibz(\s_0)=0\}.
\]
As remarked in Section \ref{sec:models},
$1-e^{-2\bc} = \psd(2)$.
It is well known that there exists a unique infinite-volume
measure for the Ising model on $\ZZ^2$ if either $h \ne0$ or
$\bbb< \bc$, and thus $\pibh$ is this measure. By Holley's theorem,
(see \cite{G-RC}, Section 2.1, e.g.), $\pibh$ is
stochastically increasing in $h$.

Let
\[
\ttheta^+(\bbb,h)=\pibh(0\lrap\oo),\qq\ttheta^-(\bbb,h) =\pibh(0\lram\oo),
\]
where the
relation $\lrap$ (\resp, $\lram$) means that there exists a path of
$\ZZ^2$ each of whose vertices
has state $+1$ (\resp, a $*$-connected path of vertices
with state $-1$). The next theorem states the absence of coexistence of such
infinite components, and its proof (given in Section \ref{sec:isingpf}) is a simple application
of the Zhang argument for percolation (see \cite{G99}, Section 11.3).

\begin{theorem}\label{zhangisi}We have that
\[
\ttheta^+(\bbb,h)\ttheta^-(\bbb,h) = 0,\qq\bbb\ge0,\ h\ge0.
\]
\end{theorem}

There exists
$\hc=\hc(\bbb)\in[0,\oo)$ such that
\[
\ttheta^+(\bbb,h)\cases{
=0, &\q\mbox{if } $0 \le h < \hc$,\cr
>0, &\q\mbox{if } $h > \hc$.}
\]
Recall from \cite{Hig0,Higuchi} that $\hc(\bbb) >0$ if and only if $\bbb
<\bc$.
It is proved in \cite{Higuchi} that
%
\begin{equation}\label{mel28}
\ttheta^\pm(\bbb,\hc(\bbb))=0,
\end{equation}
but we shall not make use of this fact in the proofs of this paper.
Indeed, one of the main purposes of this article is
to show how certain sharp-thresholds for box-crossings may be obtained
using a minimum of background information on the model in question.

Let $\hkm$ be the event that there exists a left--right
$+$ crossing of the box $B_{k,m}=[0,k] \times[0,m]$.
Let $x^+=\max\{x,0\}$.

\begin{theorem}\label{isingcr}
Let $0 \le\bbb<\bc$ and $R>0$. There exist $\rho_{i,+} =\rho
_{i,+}(\bbb)$
and $\rho_{i,-} =\rho_{i,-}(\bbb,R)$
satisfying
%
\begin{equation}\label{ihp3}
\rho_{i,+}\rho_{i,-} \to0\qq\mbox{as } i\to\oo,
\end{equation}
such that: for $0\le h_1\le\hc\le h_2 < R$,
%
\begin{equation}\label{ihp1}
\pibho(\hkm)[1-\pibht(\hkm)] \le\rho_{k,+}^{\hc-h_1}
\rho_{m,-}^{h_2-\hc},\qq k,m\ge1.
\end{equation}
\end{theorem}

The proof of this theorem shows also that
\begin{eqnarray*}
\pibho(\hkm)[1-\pibht(\hkm)] &\le&\rho_{k,+}^{h_2-h_1}, \qq h_1\le h_2 \le\hc,\\
\pibho(\hkm)[1-\pibht(\hkm)] &\le&\rho_{m,-}^{h_2-h_1},\qq \hc \le h_1\le h_2.
\end{eqnarray*}

As in Theorem \ref{thm1}, the proof neither uses nor
implies any estimate on the
rate at which $\rho_{i,\pm}\to0$. The $\rho_{i,\pm}$ are related to
the tails of the radii of the $+$ cluster and the $-$ $*$-cluster at
the origin.
More explicitly,
%
\begin{eqnarray}
\rho_{i,+} &=& [2(1+e^{8\bbb})\pibhc(0\lrap\pd\La_{i/2})]^{B_+},\\
\rho_{i,-} &=& [2(1+e^{8\bbb+2R})\pibhc(0\lram\pd\La_{i/2})]^{B_-},
\end{eqnarray}
where
\[
B^+ = 2c \xi_{\bbb,\hc},\qquad B_- = 2c \xi_{\bbb,R},
\]
and $\xi_{\bbb,h}$ is given in the forthcoming \eqref{ihp10}.
Equation \eqref{ihp3} holds by Theorem \ref{zhangisi}
with $h=\hc(\bbb)$. It is in fact a consequence of \eqref{mel28}
that $\rho_{i,\pm} \to0$ as $i\to\oo$.

\section{Influence and sharp threshold}\label{sec:ist}
Let $S$ be a finite set. Let $\mu$ be a measure on
$\Om=\{0,1\}^S$ satisfying the FKG lattice
condition \eqref{4.2}, and assume that $\mu$ is \textit{positive} in that
$\mu(\omega )>0$ for all $\omega \in\Om$. It is standard that, for
a positive measure $\mu$,
\eqref{4.2} is equivalent to the condition that $\mu$ be \textit{monotone},
which is to say that the one-point conditional measure $\mu(\s
_x=1\vert\s_y=\eta_y\mbox{ for } y\ne x)$
is nondecreasing in $\eta$. Furthermore, \eqref{4.2} implies
that $\mu$ is positively associated, in that increasing events are positively
correlated. See, for example, \cite{G-RC}, Chapter 2.

For $p\in(0,1)$, let $\mu_p$ be given by
%
\begin{equation}\label{GG1}
\mu_p(\omega ) = \frac1{Z_p}
\biggl\{\prod_{s\in S} p^{\omega (s)}(1-p)^{1-\omega (s)}\biggr\}
\mu(\omega ),\qq\omega \in\Om,
\end{equation}
where $Z_p$ is chosen in such a way
that $\mu_p$ is a probability measure. It is easy to check that each
$\mu_p$
satisfies the FKG lattice condition.

Let $A$ be an increasing event, and write $1_A$ for its indicator function.
We define the (\textit{conditional}) \textit{influence}  of the element $s\in S$ on the event $A$ by
%
\begin{equation}\label{GG3}
J_{A,p}(s) = \mu_p(A \vert1_s = 1) - \mu_p(A\vert1_s=0), \qq s \in S,
\end{equation}
where $1_s$ is the indicator function that $\omega (s)=1$.
Note that $J_{A,p}(s)$ depends on the choice of $\mu$. The conditional
influence is not generally equal to the (absolute) influence of \cite{KKL},
\[
I_{A,p}(s) = \mu_p\bigl(1_A(\omega ^s) \ne1_A(\omega _s)\bigr),
\]
where the configuration $\omega ^s$
(\resp, $\omega _s$) is that obtained from $\omega $ by setting
$\omega (s)=1$ [\resp, $\omega (s)=0$].

\begin{theorem}\label{inf1}
There exists a constant $c>0$ such that the following holds.
For any such $S$, $\mu$, and any increasing event $A \ne\es,\Om$,
%
\begin{equation}\label{vanc11}
\frac d{dp}\mu_p(A) \ge\frac{c\xi_p}{p(1-p)}\mu_p(A)\bigl(1-\mu_p(A)\bigr)\log[1/(2m_{A,p})],
\end{equation}
where $m_{A,p} = \max_{s\in S} J_{A,p}(s)$ and
$\xi_p= \min_{s\in S}[\mu_p(1_s)(1-\mu_p(1_s))]$.
\end{theorem}

\begin{cor}\label{inf2}
In the notation of Theorem \ref{inf1},
\[
\mu_{p_1}(A)[1-\mu_{p_2}(A)] \le\kappa^{B(p_2-p_1)}, \qq 0<p_1\le p_2<1,
\]
where
\[
B=\inf_{p\in(p_1,p_2)} \biggl\{\frac{c\xi_p}{p(1-p)}\biggr\},\qq
\kappa= 2\mathop{\mathop{\operatorname{sup}}_{{p\in(p_1,p_2),}}}_{s\in S} J_{A,p}(s).
\]
\end{cor}

The corresponding inequality for product measures may be found
in \cite{Tal94}, Corollary 1.2. Throughout this note, the letter $c$ shall refer
only to the constant of Theorem \ref{inf1}.

\begin{pf*}{Proof of Theorem \ref{inf1}}
It is proved in \cite{BGK,GG} that
%
\begin{equation}\label{mel1}
\frac d{dp}\mu_p(A) = \frac1 {p(1-p)} \sum_{s\in S} \mu
_p(1_s)\bigl(1-\mu_p(1_s)\bigr) J_{A,p}(s).
\end{equation}

Let $K=[0,1]^S$ be the ``continuous'' cube, endowed with Lebesgue measure
$\lambda$,
and let $B$ be an increasing subset of $K$. The influence $I_B(s)$ of
an element $s$
is given in \cite{BKKKL} as
\[
I_B(s) = \lambda\bigl(1_B(\psi^s) \ne1_B(\psi_s)\bigr),
\]
where $\psi^s$ (\resp, $\psi_s)$ is the member of $K$ obtained from
$\psi\in K$
by setting $\psi(s)=1$ [\resp, $\psi(s)=0$]. The conclusion of \cite{BKKKL}
may be
expressed as follows. There exists a constant $c>0$, independent of
all other quantities, such
that: for any increasing event $B \subseteq K$,
%
\begin{equation}\label{bkkkl1}
\sum_{s\in S} I_B(s) \ge c \lambda(B)\bigl(1-\lambda(B)\bigr) \log[1/(2m_B)],
\end{equation}
where $m_B = \max_{s\in S} I_B(s)$. The main result of
\cite{BKKKL} is a lower bound on $m_B$ that is easily seen to follow
from \eqref{bkkkl1}.

Equation \eqref{bkkkl1} does not in fact appear
explicitly in \cite{BKKKL}, but it may be derived from the arguments presented
there, very much as observed in the case of the discrete cube
from the arguments of \cite{KKL}. See \cite{FKST}, Theorem 3.4.
The factor of 2
on the right-hand side of \eqref{bkkkl1} is of little material
consequence, since the inequality is important only when $m_B$ is
small, and, when $m_B < \frac13$ say, the 2 may be removed with an
amended value of the constant $c$.
The literature on influence and sharp-threshold can seem a little disordered,
and a coherent account may be found in \cite{G-pgs}. The method
used there introduces the factor 2 in a natural way, and for
this reason we have included it in the above.

It is shown in \cite{GG} (see the proof of Theorem 2.10)
that there exists an increasing subset $B$ of $K$ such that
$\mu_p(A)=\lambda(B)$, and $J_{A,p}(s) \ge I_B(s)$ for all $s \in S$.
Inequality \eqref{vanc11} follows by \eqref{mel1}--\eqref{bkkkl1}.
\end{pf*}

\begin{pf*}{Proof of Corollary \ref{inf2}}
By \eqref{vanc11},
\[
\biggl(\frac1{\mu_p(A)} + \frac1{1-\mu_p(A)}\biggr)\mu_p'(A)
\ge B\log(\kappa^{-1}),\qq p_1 < p < p_2,
\]
whence, on integrating over $(p_1,p_2)$,
\[
\frac{\mu_{p_2}(A)}{1-\mu_{p_2}(A)}\Big/\frac{\mu
_{p_1}(A)}{1-\mu_{p_1}(A)}
\ge\kappa^{-B(p_2-p_1)}.
\]
The claim follows.
\end{pf*}

\section[Proofs of Theorems 3.1 and 3.2]{Proofs of Theorems \protect\ref{thm1}
and \protect\ref{thm3}}\label{sec:pf}

Note first that a \rc\ measure has the form of \eqref{GG1} with $S=E$
and $\mu(\omega )=q^{k(\omega )}$,
and it is known and easily checked that $\mu$ satisfies the FKG
lattice condition
when $q \ge1$ (see \cite{G-RC}, Section~3.2, e.g.). We shall
apply Theorem \ref{inf1}
to a \rc\ $\fpq$ measure with $q \ge1$.
It is standard (see \cite{G-RC}, Theorem 4.17(b)) that
%
\begin{equation}\label{mel29}
\frac pq \le\frac p{p+q(1-p)} \le\fpq(1_e) \le p,
\end{equation}
whence
\[
\fpq(1_e)[1-\fpq(1_e)] \ge\frac{p(1-p)}q.
\]
We may thus take
%
\begin{equation}\label{mel4}
B= \frac cq
\end{equation}
in Corollary \ref{inf2}.

Let $q \ge1$,
$1\le k,m<n$, and consider the \rc\ measures $\fnp^b = \fnpq^b$
on the box $\La_n$.
For $e \in\EE^2$, write $\jkmn^b(e)$ for the (conditional) influence
of $e$ on
the event $\hkm$ under the measure $\fnp^b$.
We set $\jkmn^b(e)=0$ for $e \notin\EE_{\La_n}$.

\begin{lem}\label{lemma1}
Let $q\ge1$. We have that
%
\begin{eqnarray}
\operatorname{sup}\limits_{e\in\EE^2} \jkmn^0(e) & \le & \frac qp \eta_k, \qq\hspace*{17.5pt} 0<p\le\kq,\ 1\le k,m < n,\\
\operatorname{sup}\limits_{e\in\EE^2} \jkmn^1(e) & \le & \frac q{p_{\rd}} \eta_{m+1},\qq \kq\le p<1,\ 1\le k,m < n,
\end{eqnarray}
where $p_{\rd}$ satisfies \eqref{dualv} and
\[
\eta_k = \phi_{\kq,q}^0(0\lra\pd\La_{k/2}) \to0\qq\mbox{as }k\to\oo.
\]
\end{lem}

\begin{pf}
For any configuration $\omega \in\Om$ and vertex $z$,
let $C_z(\omega )$ be the open cluster at $z$,
that is, the set of all
vertices joined to $z$ by open paths.

Suppose first that $0<p\le\kq$, and let $e = \langle x,y\rangle$ be
an edge
of $\La_n$. We couple the two conditional measures $\fnp^0(\cdot\vert
\omega (e)=b)$,
$b=0,1$, in the following manner. Let $\Om_n$ be the configuration space
of the edges in $\La_n$, and let $T=\{(\pi,\omega )\in\Om_n^2\dvtx \pi
\le\omega \}$
be the set of all ordered pairs of configurations. There exists a measure
$\mu^e$ on $T$ such that:
\begin{longlist}
\item[(a)] the first marginal of $\mu^e$ is $\fnp^0(\cdot\vert 1_e=0)$,
\item[(b)] the second marginal of $\mu^e$ is $\fnp^0(\cdot\vert 1_e=1)$,
\item[(c)] for any subset $\ggamma$ of $\La_n$, conditional on the event
$\{(\pi,\omega )\dvtx C_x(\omega )=\ggamma\}$,
the configurations $\pi$ and $\omega $ are $\mu^e$-almost-surely
equal on all edges having no endvertex in $\ggamma$.
\end{longlist}
The details of this coupling are omitted. The idea is to build the paired
configuration $(\pi,\omega )$ edge by edge, beginning at the edge $e$,
in such a way that $\pi(f) \le\omega (f)$ for each edge $f$ examined.
The (closed) edge-boundary of the cluster $C_x(\omega )$ is closed in
$\pi$ also. Once this boundary has been uncovered, the configurations
$\pi$, $\omega $
on the rest of space are governed by the same (conditional) measure,
and may
be taken equal. Such an argument has been used in \cite{ACCN} and
\cite{G-RC}, Theorem 5.33(a),
and has been carried further in \cite{Al1}.

We claim that
%
\begin{equation}\label{mel30}
\jkmn^0(e) \le\fnp^0(D_x \vert 1_e=1),
\end{equation}
where $D_x$ is the event that $C_x$ intersects both the left and right sides
of $\bkm$. This is proved as follows. By \eqref{GG3},
\begin{eqnarray*}
\jkmn^0(e) &=& \mu^e(\omega \in\hkm,\pi\notin\hkm)\\
&\le& \mu^e(\omega \in\hkm\cap D_x)\\
&\le& \mu^e(\omega \in D_x)=\fnp^0(D_x \vert 1_e=1),
\end{eqnarray*}
since, when $\omega \notin D_x$, either both or neither of $\omega $,
$\pi$ belong to $\hkm$.
By \eqref{mel30},
%
\begin{equation}\label{mel41}
\jkmn^0(e) \le\frac{\fnp^0(D_x)}{\fnp^0(1_e)}.
\end{equation}

On $D_x$, the radius of the open cluster at $x$
is at least $\frac12 k$. Since $\fnp^0 \lest\fpq$ and $\fpq$ is
translation-invariant,
\[
\fnp^0(D_x) \le\fpq(x \lra x + \pd\La_{k/2})
= \fpq(0\lra\pd\La_{k/2}).
\]

By \eqref{theta=0},
\[
\fpq(0\lra\pd\La_{k/2}) \le\phi_{\kq,q}^0(0\lra\pd\La_{k/2})\to 0 \qq \mbox{as } k \to\oo,
\]
and, by \eqref{mel29} and \eqref{mel41}, the conclusion of the lemma is
proved when $p \le\kq$.

Suppose next that $\kq\le p<1$. Instead of working with the open
paths, we work
with the dual open paths. Each edge $e_{\rd}=\langle u,v\rangle$
of the dual lattice
traverses some edge
$e=\langle x,y\rangle$ of the primal,
and, for each configuration $\omega $, we define
the dual configuration $\omega _{\rd}$ by
$\omega _{\rd}(e_{\rd}) = 1-\omega (e)$. Thus, the dual edge $e_{\rd}$
is open
if and only if $e$ is closed. It is well known
(see \cite{G-RC}, Equation (6.12), e.g.) that, with $\omega $ distributed according to $\fnp^{1}$,
$\omega_{\rd}$ has as law the \rc\ measure, denoted $\fnpd$, on the dual of $\La_n$ with
free boundary condition.
The event $\hkm$ occurs if and only if there
is no dual open path traversing the dual of $\bkm$ from top to bottom.
We may therefore apply the above argument to the dual process,
obtaining thus that
%
\begin{equation}\label{mel51}
\jkmn^1(e) \le\frac{\fnpd(V_u)}{\fnpd(1_e)},
\end{equation}
where $V_u$ is the event that $C_u$ intersects both the top and bottom
sides of the dual of $\bkm$.

On the event $V_u$, the radius of the open cluster at $u$
is at least $\frac12 (m+1)$. Since $\fnpd\lest\fpdq$,
\[
\fnpd(V_u) \le\fpdq\bigl(u \lra u + \pd\La_{(m+1)/2}\bigr) = \fpdq\bigl(0\lra\pd\La_{(m+1)/2}\bigr).
\]

As above, by \eqref{mel60},
\[
\fpdq\bigl(0\lra\pd\La_{(m+1)/2}\bigr) \le\phi_{\kq,q}^0\bigl(0\lra\pd\La_{(m+1)/2}\bigr)
=\eta_{m+1},
\]
and this completes the proof when $p \ge\kq$.
\end{pf}

\begin{pf*}{Proof of Theorem \ref{thm3}}
This follows immediately from Corollary \ref{inf2} by \eqref{mel4}
and Lemma \ref{lemma1}.
\end{pf*}

\begin{pf*}{Proof of Theorem \ref{thm1}}
By planar duality,
\[
\phi_{p,q}^0(H_k) = 1- \phi_{p_{\rd},q}^1(H_k),
\]
where $p$, $p_{\rd}$ are related by \eqref{dualv},
see \cite{G-RC}, Theorems 6.13, 6.14. Since $\phi_{\kq,q}^0 \lest\phi_{\kq,q}^1$,
\[
\phi_{\kq,q}^0(H_k) \le\tfrac12 \le\phi_{\kq,q}^1(H_k),
\]
and Theorem \ref{thm1} follows from Theorem \ref{thm3}.
\end{pf*}

\section[Proof of Theorems 4.1 and 4.2]{Proof of Theorems \protect\ref{zhangisi}
and \protect\ref{isingcr}}\label
{sec:isingpf}

Only an outline of the proof of Theorem \ref{zhangisi} is included
here, since it follows the
``usual'' route (see \cite{G99}, Section~11.3, or \cite{G-RC}, Section 6.2,
 for examples
of the argument). The measure $\pibh$ is automorphism-invariant,
ergodic, and has the finite-energy property.
By the main result of \cite{BK}, the number $N^+$ (\resp, $N^-$) of
infinite $+$ clusters
(\resp, infinite $-$ $*$-connected clusters) satisfies
\[
\mbox{either}\q\pibh(N^\pm= 0) = 1\q\mbox{or}\q\pibh(N^\pm=1)=1.
\]
Assume that $\ttheta^+(\bbb,h)\ttheta^-(\bbb,h) > 0$, which is to say that
$\pibh(N^+=N^-=1) =1$. One may find a box $\La$ sufficiently large that,
with $\pibh$-probability at least $\frac12$: the
top and bottom of its boundary $\pd\La$ are $+$ connected to infinity
off $\La$,
and the left and right sides are $-$ $*$-connected to infinity off $\La$.
Since $N^+=1$ almost surely, there is a $+$ path connecting the two
infinite $+$ paths above,
and this contradicts the fact that $N^-=1$ almost surely.

We turn to the proof of Theorem \ref{isingcr}.
For the moment, let $\pibh$ be the Ising measure on a finite graph
$G=(V,E)$ with parameters $\bbb\ge0$
and $h \ge0$. It is well known than $\pi_{\bbb,0}$ satisfies the
FKG lattice condition \eqref{4.2} on the
partially ordered set $\Si_V = \{-1,+1\}^V$. We identify
$\Si_V$ with $\{0,1\}^V$ via the mapping $\s_x \mapsto\omega _x =
\frac12(\s_x+1)$,
and we choose $p$ by
%
\begin{equation}\label{mel40}
\frac p{1-p} = e^{2h}.
\end{equation}
Then $\pibh$ may be expressed in the form \eqref{GG1}, and we
may thus apply the results of Section \ref{sec:ist}.
By conditioning on the states of the neighbors of $x$,
%
\begin{equation}\label{mel32}
\dfrac{e^{2h-\De\bbb}}{e^{\De\bbb}+e^{2h-\De\bbb}} \le\pibh(1_x) \le
\dfrac{e^{2h+\De\bbb}}{e^{-\De\bbb}+e^{2h+\De\bbb}} ,
\end{equation}
where $\De$ is the degree of the vertex $x$, and $1_x$ is the
indicator function
that $\s_x=+1$. Therefore,
\begin{eqnarray}\label{mel33}
\quad\pibh(1_x)[1-\pibh(1_x)]
&\ge&\min\biggl\{ \dfrac{e^{2h}}{(e^{\De\bbb}+e^{2h-\De\bbb})^2}, \dfrac
{e^{2h}}{(e^{-\De\bbb}+e^{2h+\De\bbb})^2}
\biggr\}\nonumber\\[-8pt]\\[-8pt]
&=& \frac{e^{2h+2\De\bbb}}{(1+e^{2h+2\De\bbb})^2}.\nonumber
\end{eqnarray}
This bound will be useful with $\De= 4$, and we write
%
\begin{equation}\label{ihp10}
\xi_{\bbb,h} = \frac{e^{2h+8\bbb}}{(1+e^{2h+8\bbb})^2}.
\end{equation}
Note that $\xi_{\bbb,h}$ is decreasing in $h$.

We follow the argument of the proof of Theorem \ref{inf1}.
Let $\bbb\in[0,\bc)$, $h >0$,
and $1\le k,m \le r <n$, and consider the Ising measure $\pin=\pi_{\La_n,\bbb,h}$
on the box $\La_n = [-n,n]^2$.
For $x\in\ZZ^2$, write $\jkmn(x)$ for the (conditional) influence of
$x$ on
the event $\hkm$ under the measure $\pin$. We set $\jkmn(x)=0$ for
$x \notin\La_n$.

\begin{lem}\label{lemma2}
Uniformly in $x\in\ZZ^2$,
\begin{eqnarray}
&&\jkmn(x) \le (1+e^{8\bbb-2h})\nonumber\label{ihp4}\\[-8pt]\\[-8pt]
&&\hphantom{\jkmn(x) \le }{}\times\Bigl[\pin(\bkm\lrap\pd\La_r) +
\operatorname{sup}\limits_{x\in\La_r}\pin(x\lrap x +\pd\La_{k/2})\Bigr],\nonumber\\
&&\jkmn(x) \le (1+e^{8\bbb+2h})\nonumber\label{ihp5}\\[-8pt]\\[-8pt]
&&\hphantom{\jkmn(x) \le }{}\times\Bigl[\pin(\bkm\lram\pd\La_r) +
\operatorname{sup}\limits_{x\in\La_r}\pin(x\lram x + \pd\La_{m/2})\Bigr].\nonumber
\end{eqnarray}
\end{lem}

\begin{pf}
Let $h > 0$.
Let $C_x^+$ be the set of all vertices joined to $x$ by a path
of vertices all of whose states are $+1$ (thus, $C_x^+=\es$ if $\s_x=-1$).
We may couple the conditioned measures $\pin(\cdot\vert\s_x=b)$,
$b= \pm1$, such that the Ising equivalents of (a)--(c) hold as in
Section \ref{sec:pf}.
As in \eqref{mel41},
%
\begin{equation}\label{mel31}
\jkmn(x) \le\frac{\pin(D_x^+)}{\pin(1_x)},
\end{equation}
where $D_x^+$ is the event that $C_x^+$ intersects both the left and
right sides
of $\bkm$.
On $D_x^+$, the radius of $C_x^+$
is at least $\frac12 k$.

For $x \notin\La_r$,
\[
\pin(D_x^+) \le\pin(\bkm\lrap\pd\La_r).
\]
For $x \in\La_r$, we shall use the bound
\[
\pin(D_x^+) \le\pin(x\lrap x + \pd\La_{k/2}).
\]
Combining the above inequalities with \eqref{mel32}, we obtain
\eqref{ihp4}.

Let $C_x^-$ be the
set of all vertices joined to $x$ by a $*$-connected path
of vertices all of whose states are $-1$.
The event $\hkm$ occurs
if and only if there is no $-$ $*$-connected path
from the top to the bottom of $\bkm$. Therefore, the conditional influence
of $x$ on $\hkm$ equals that of $x$ on this new event. As in \eqref{mel31},
%
\begin{equation}\label{mel31b}
\jkmn(x) \le\frac{\pin(V_x^-)}{\pin(1-1_x)},
\end{equation}
where $V_x^-$ is the event that $C_x^-$ intersects both the top and bottom
of $\bkm$.
The above argument leads now to \eqref{ihp5}.
\end{pf}

\begin{pf*}{Proof of Theorem \ref{isingcr}}
Let $R>\hc$ and $\delta >0$, and let $k,m\le r< n$. We set
\begin{eqnarray*}
\kappa_{n,r,+}^\delta &=& 2(1+e^{8\bbb})\Bigl[\pinhcdm(\bkm\lrap\pd
\La_r)+\sup_{x\in\La_r}\pinhcdm(x\lrap x + \pd\La_{k/2})\Bigr],\\
\kappa_{n,r,-}^\delta &=& 2(1+e^{8\bbb+2R})\\
&&{}\times\Bigl[\pinhcdp(\bkm\lram\pd\La_r)+\sup_{x\in\La_r}\pinhcdp(x\lram x +
\pd\La_{m/2})\Bigr].
\end{eqnarray*}

Let $0<h_1 < \hc< h_2\le R$, and choose $\delta < \min\{\hc-h_1,
h_2-\hc\}$.
By \eqref{mel40}, \eqref{mel33}, Lemma \ref{lemma2} and Theorem
\ref{inf1},
$f_n(h)=\pin(\hkm)$ satisfies
%
\begin{equation}\label{ihp7}
\frac1{f_n(h)(1-f_n(h))}\cdot\frac{df_n}{dh} \ge B_+ \log(1/\kappa
_{n,r,+}^\delta ),\qq h_1\le h \le\hc-\delta ,
\end{equation}
where
$B_+ = 2c \xi_{\bbb,\hc}$,
see \eqref{ihp10}.
The corresponding inequality for $\hc+\delta \le h\le R$ holds
with $\kappa_{n,r,+}^\delta $ replaced by $\kappa_{n,r,-}^\delta $,
and $B_+$ replaced by $B_- = 2c\xi_{\bbb,R}$.

We integrate \eqref{ihp7} over the intervals $(h_1,\hc-\delta )$
and $(\hc+\delta ,h_2)$, add the results, and use the fact that $f_n(h)$
is nondecreasing in $h$, to obtain that
\begin{eqnarray*}
&&\log\frac{f_n(h)}{1-f_n(h)}\bigg|_{h_1}^{h_2}
\ge(\hc-\delta -h_1)B_+\log(1/\kappa_{n,r,+}^\delta )\\
&&\hphantom{\log\frac{f_n(h)}{1-f_n(h)}\bigg|_{h_1}^{h_2}
\ge}{} + (h_2-\hc
-\delta )B_-\log(1/\kappa_{n,r,-}^\delta ).
\end{eqnarray*}
Take the limits as $n\to\oo$, $r\to\oo$, and $\delta \to0$ in
that order,
and use the monotonicity in $h$ of $\pibh$, to obtain the theorem.
\end{pf*}

\section{The colored \rc\ model}\label{sec:crcm}
There is a well known coupling of the \rc\ and Potts models that
provides a transparent explanation of how the analysis of the former
aids that of the latter. Formulated as in \cite{ES} (see also the
historical account of \cite{G-RC}), this is as follows. Let $p\in
(0,1)$ and $q\in\{2,3,\dots\}$. Let $\omega $ be sampled from the
\rc\ measure $\fpq$ on the finite graph $G=(V,E)$. To each open
cluster of $\omega $, we assign a uniformly chosen element of
$\{1,2,\dots,q\}$, these random spins being independent between
clusters. The ensuing spin-configuration $\s$ on $G$ is governed by a
Potts measure, and pair-spin correlations in $\s$ are coupled to open
connections in $\omega $. This coupling has inspired a construction that
we describe next.

Let $p\in(0,1)$, $q \in(0,\oo)$, and $\aalpha\in(0,1)$. Let $\omega $ have
law $\fpq$. To the vertices of each open cluster of $\omega $, we
assign a
random spin chosen according to the Bernoulli measure on $\{0,1\}$
with parameter $\aalpha$. These spins are constant within clusters, and
independent between clusters. We call this the \textit{colored
\rc\ model} (CRCM). With $\s$ the ensuing spin-configuration, we
write $\kpqa$ for the measure governing the pair $(\omega ,\s)$, and
$\ppqa$ for the marginal law of $\s$. When $q\in\{2,3,\dots\}$ and
$q\aalpha$ and $q(1-\aalpha)$ are integers, the CRCM is a vertex-wise
contraction of the Potts model from the spin-space $\{1,2,\dots,q\}^V$
to $\Si= \{0,1\}^V$.

The CRCM has been studied in \cite{KahnW} under the name ``fractional
fuzzy Potts model,'' and it is inspired in part by the earlier work of
\cite{Chay96,Hag99,Hag01}, as well as the study of the so-called
``divide-and-colour model'' of \cite{DaCmodel}.

The following seems to be known, see \cite{Chay96,Hag99,Hag01,KahnW},
but the short proof given below may be of value.

\begin{theorem}\label{crcfk}
The measure $\ppqa$ is monotone for all finite graphs $G$ and all
$p\in(0,1)$ if and only if $q\aalpha,q(1-\aalpha)\ge1$.
\end{theorem}

We identify the spin-vector $\s\in\Si$ with the set $A=\{v\in V\dvtx
\s_v=1\}$. Let $\ph=\ppqah$ be the probability measure obtained from
$\ppqa$ by including an external field with strength $h\in\RR$,
\begin{equation}\label{ihp20}
\ph(A)\propto e^{h|A|}\ppqa(A), \qq A \subseteq V.
\end{equation}
It is an elementary consequence of Theorem \ref{crcfk} and
\eqref{ihp20} that, when $q\aalpha, q(1-\aalpha) \ge1$, $\ph$ is a monotone
measure, and $\ph$ is increasing in $h$. When $q=2$ and $\aalpha=\frac12$,
$\ph$ is the Ising measure with external field. The purpose of this
section is to extend the arguments of Section \ref{sec:ising} to the
CRCM with external field.

There is a special case of the CRCM with an interesting
interpretation. Let $\omega $ be sampled from $\fpq$ as above, and let
$\s=(\s_v\dvtx v \in V)$ be a vector of independent Bernoulli ($\ggamma$)
variables. Let $B$ be the event that $\s$ is constant on each open
cluster of $\omega $. The pair $(\omega ,\s)$, conditional on $B$,
is termed
the \textit{massively colored \rc\ measure} (MCRCM). The law of $\s$ is
simply $\pi_{p,2q,1/2,h}$ where $h=\log[\ggamma/(1-\ggamma)]$.

Just as $\ppqa$ and $\fpq$ may be coupled via $\kpqa$, so we can
couple $\ph$ with an ``edge-measure'' $\psh=\pspqahh$ via the following
process. With $B$ given as above, and $(\omega ,\s)\in B$, denote by
$\s(C)$ the common spin-value of $\s$ on an open cluster $C$ of
$\omega $. Let $\kh=\kpqah$ be the probability measure on $\Om
\times\Si$ given by
\begin{equation}\label{ihp22}
\kh(\omega ,\s) \propto\fpq(\omega
)1_B(\omega ,\s)\prod_C
\bigl[\bigl(\aalpha e^{h|C|}\bigr)^{\s(C)}(1-\aalpha)^{1-\s(C)}\bigr],
\end{equation}
where the product is over the open clusters $C$ of $\omega $, and
$|C|$ is the number of vertices of $C$. The marginal and conditional
measures of $\kh$ are easily calculated. The marginal on $\Si$ is
$\ph$, and the marginal on $\Om$ is $\psh= \pspqahh$ given by
%
\begin{equation}\label{crc4}
\psh(\omega )\propto\phi_{p,q}(\omega ) \prod_C \bigl[ \aalpha
e^{h|C|}+1-\aalpha\bigr],\qq\omega \in\Om.
\end{equation}
Note that $\phi_0 = \fpq$. Given $\omega $, we obtain $\s$ by labeling
the open clusters with independent Bernoulli spins in such a way that
the odds of cluster $C$ receiving spin $1$ are $\aalpha e^{h|C|}$ to
$1-\aalpha$.

By \eqref{ihp20}, or alternatively by summing $\kh(\omega ,\s)$ over
$\omega $, we find that
\begin{equation}\label{crc1}
\ph(A) \propto e^{h|A|} (1-p)^{|\De A|}
Z_{A,q\aalpha}Z_{\ol A,q(1-\aalpha)} , \qq A \subseteq V,
\end{equation}
where $\De A$ is the set of edges of $G$ with exactly one
endvertex in $A$, and $Z_{B,q}$ is the partition function of the
\rc\ measure on the subgraph induced by $B\subseteq V$ with
edge-parameter $p$ and cluster-weight $q$. It may be checked as in the
proof of Theorem \ref{crcfk} that, for given $p$, $q$, $\aalpha$, the measure
$\ph$ is bounded above (\resp, below) by a product measure with parameter $a(h)$ [\resp, $b(h)$]
where
\begin{equation}\label{ihp24}
 a(-h) \to0,\qq b(h) \to1, \qq\mbox{as } h\to\oo.
\end{equation}

The measure $\psh$ has a number of useful properties, following.

\begin{prop}\label{phi_properties}
Let $q\aalpha,q(1-\aalpha)\ge1$.
\begin{longlist}[(iii)]
\item[(i)] The probability measure $\psh$ is monotone.
\item[(ii)] The marginal measure of $\kh$ on $\Om$, conditional on $\s_x=b$,
satisfies
\begin{eqnarray*}
\kh(\cdot\vert\s_x=1)&\gest&\kh(\cdot\vert\s_x=0),\qq h\ge 0,\\
\kh(\cdot\vert\s_x=1)&\lest&\kh(\cdot\vert\s_x=0),\qq h\le 0.
\end{eqnarray*}
\item[(iii)] If $p_1\le p_2$ and the ordered three-item sequence $(0,h_1,h_2)$ is
monotonic, then $\phi_{p_1,q,\aalpha,h_1}\lest\phi_{p_2,q,\aalpha,h_2}$.
\item[(iv)] We have that $\pspqahh\lest\phi_{p,Q}$, where $Q=Q(h)$ is
defined by
\[
Q(h)=\cases{
q\aalpha, &\q $h>0$,\cr
q, &\q $h=0$,\cr
q(1-\aalpha), &\q $h<0$.}
\]
\end{longlist}
\end{prop}

\textit{We assume henceforth that $q\aalpha,q(1-\aalpha)\ge1$}, and we consider
next the infinite-volume limits of the above measures. Let $G$ be a
subgraph of the square lattice $\ZZ^2$ induced by the vertex-set $V$,
and label the above measures with the subscript $V$. By standard
arguments (see \cite{G-RC}, Chapter 4), the limit measure
\[
\psh= \lim_{V\uparrow\ZZ^2} \pshV
\]
exists, is independent of the choice of the $V$, and is
translation-invariant and ergodic. By an argument similar to that of
\cite{G-RC}, Theorem 4.91, the measures $\phV$ have a well-defined
infinite-volume limit $\ph$ as $V\uparrow\ZZ^2$. Furthermore, the
pair $(\psh,\ph)$ may be coupled in the same manner as on a finite
graph. That is, a \textit{finite} cluster $C$ of $\omega $ receives spin
$1$ with probability $\aalpha e^{h|C|}/[\aalpha e^{h|C|} + 1-\aalpha]$. An
\textit{infinite} cluster receives spin $1$ (\resp, $0$) if $h>0$
(\resp, $h<0$). When $h=0$, the spin of an infinite cluster has the
Bernoulli distribution with parameter $\aalpha$.

Since $\psh$ is translation-invariant, so is $\ph$. As in \cite{G-RC}, Theorem 4.10, $\ph$ is positively associated, and the proof of
\cite{G-RC}, Theorem 4.91, may be adapted to obtain that $\ph$ is
ergodic. By a simple calculation, the $\phV$ have the finite-energy
property, with bounds that are uniform in $V$ (see \cite{G-RC}, Equation (3.4)),
and therefore so does $\ph$. Adapting the notation used in Section \ref{sec:ising}
 for the Ising model, let
\begin{eqnarray*}
\ttheta^1(p,q,\aalpha,h)&=&\ph(0\lrab\oo),\\
\ttheta^0(p,q,\aalpha,h)&=&\ph(0\lraw\oo).
\end{eqnarray*}
As in Theorem \ref{zhangisi}, and with an essentially identical proof,
\begin{equation}\label{ihp21}
\ttheta^1(p,q,\aalpha,h)\ttheta^0(p,q,\aalpha,h) = 0.
\end{equation}
By the remark after \eqref{ihp20} and \cite{G-RC}, Theorem 4.10, $\ph$
is stochastically increasing in $h$, whence there exists
$\hc=\hc(p,q,\aalpha)\in\RR\cup\{\pm\oo\}$ such that
\[
\ttheta^1 (p,q,\aalpha,h)\cases{
 =0, &\quad if $h < \hc$,\cr
 >0, &\quad if $h > \hc$.}
\]
By comparisons with product measures [see the remark prior to
\eqref{ihp24}], we have that $|\hc|<\oo$.

We call a probability measure $\mu$ on $\Si$
\textit{subcritical} (\resp, \textit{supercritical}) if the
$\mu$-probability of an infinite $1$-cluster is
$0$ (\resp, strictly greater than $0$); we shall use the corresponding
terminology for
measures on $\Om$. There is a second type
of phase transition, namely the onset of percolation in the measure
$\psh$. An infinite edge-cluster under $\psh$ forms part of an
infinite vertex-cluster under $\ph$. Let $\pc(q)$ be the critical
point of the \rc\ measure $\fpq$ on $\ZZ^2$, as usual. By Proposition~\ref{phi_properties}(iv), $\psh$ is subcritical for all $h$ when
$p<\pc(q\min\{\aalpha,1-\aalpha\})$; in particular, for such $p$, $\psh$ is
subcritical for $h$ lying in some open neighborhood of $\hc$. On the
other hand, suppose that $\phi_0=\fpq$ is supercritical. By the
remarks above, $\ttheta^1>0$ for $h>0$, and $\ttheta^0>0$ for $h<0$. By
\eqref{ihp21}, $\ttheta^1$ is discontinuous at $h=\hc=0$. By Proposition
\ref{phi_properties}(iii), $\psh\gest\phi_0$, whence $\ttheta^1$ is
discontinuous at $h=\hc=0$ whenever $p>\pc(q)$.

With $k,m\in\NN$, let $H_{k,m}$ be the event that there exists a
left--right $1$-crossing of the box $B_{k,m}$. A result corresponding to
Theorem \ref{isingcr} holds, subject to a condition on $\psh$ with $h$
near $\hc$. This condition has not, to our knowledge, been verified for
the Ising model, although it is expected to hold. In this sense, the
next theorem
does not quite generalize Theorem \ref{isingcr}.

\begin{theorem}\label{crcmcr}
Let $R\ge0$. When $\hc\not=0$, we require in addition that $R\le
|\hc|$.
Suppose that $\psh$ is subcritical for $h\in[\hc-R,\hc+R]$. There exist
$\rho_{i,1}=\rho_{i,1}(p,q,\aalpha,R)$ and
$\rho_{i,0}=\rho_{i,0}(p,q,\aalpha,R)$ satisfying
\[
\rho_{i,1}\rho_{i,0} \to 0 \qq\mbox{as } i\to\oo,
\]
such that: for $h_1\in[\hc-R,\hc]$,
$h_2\in[\hc,\hc+R]$,
\[
\pi_{h_1}(\hkm)[1-\pi_{h_2}(\hkm)] \le\rho_{k,1}^{\hc-h_1}
\rho_{m,0}^{h_2-\hc},\qq k,m\ge1.
\]
\end{theorem}

As in the proof of Theorem \ref{isingcr}, the first step is to
establish bounds on the one-point marginals of $\ph$. This may be
strengthened to a finite-energy property, but this will not be required here.
The proof is deferred to the end of the section.

\begin{lem}\label{crcm_finite_energy}
Let $G=(V,E)$ be a finite graph with maximum vertex-degree $\De$.
Then
\[
\frac{\aalpha e^h}{\aalpha e^h +1 -\aalpha}\, (1-p)^\De\le\ph(\s_x=1)\le
1-\frac{1-\aalpha}{\aalpha e^h +1 -\aalpha}\, (1-p)^\De.
\]
\end{lem}

Consider the subgraph of $\ZZ^2$ induced by $\La_n=[-n,n]^d$, and let
$x\in\La_n$. Objects associated with the finite domain $\La_n$ are
labeled with the subscript $n$. For $b=0,1$, let $\pnh^b$ (\resp,
$\pnsh^b$) be the marginal measure on $\Si_n$ (\resp, $\Om_n$) of the
coupling $\knh$ conditioned on $\s_x=b$.

By Proposition \ref{phi_properties}, $\pnsh^1 \gest\pnsh^0$ when $h
\ge0$, and $\pnsh^1 \lest\pnsh^0$ when $h \le0$. It is convenient
to work with a certain coupling of the pairs $(\pnsh^0,\pnh^0)$ and
$(\pnsh^1,\pnh^1)$. Recall that $C_x(\omega )$ denotes the open
cluster at
$x$ in the edge-configuration $\omega \in\Om$.

\begin{lem}\label{coupling}
Let $h\in\RR$. There exists a probability measure $\kappa^{01}_{n,h}$
on $(\Om_n\times\Si_n)^2$ with the following properties. Let
$(\omega ^0,\s^0,\omega ^1,\s^1)$ be sampled from $(\Om_n\times
\Si_n)^2$
according to $\kappa^{01}_{n,h}$.
\begin{longlist}[(iv)]
\item[(i)] For $b=0,1$, $\omega ^b$ has law $\pnsh^b$.
\item[(ii)] For $b=0,1$, $\s^b$ has law $\pnh^b$.
\item[(iii)] If $h\le0$, $\omega ^0\ge\omega ^1$. If $h\ge0$, $\omega
^1\ge\omega ^0$.
\item[(iv)] The spin configurations $\s^0$ and $\s^1$ agree at all vertices
$y \notin C_x(\omega ^0)\cup C_x(\omega ^1)$.
\end{longlist}
\end{lem}

\begin{pf}
Assume first that $h \ge0$. There exists a probability measure
$\ol\phi_n$ on $\Om_n^2$, with support $D_1=\{(\omega ^0,\omega
^1)\in\Om_n^2\dvtx\omega ^0\le\omega ^1\}$, whose first (\resp, second) marginal is
$\pnsh^0$
(\resp, $\pnsh^1$). By sampling from $\ol\phi_n$ in a sequential
manner \vspace{1pt} beginning at $x$, and proceeding via the open connections of
the upper configuration, we may assume in addition that
$(\omega ^0,\omega ^1)\in D_2$, where $D_2$ is the set of pairs \vspace{1pt}such that
$\omega ^0(e)=\omega ^1(e)$ for any edge $e$ having at most one
endpoint in
$C_x(\omega ^1)$. Let $(\omega ^0,\omega ^1)\in D=D_1\cap D_2$.

The spin vectors $\s^b$ may be constructed as follows:
\begin{longlist}
\item[(a)] attach spin $b$ to the cluster $C_x(\omega ^b)$,
\item[(b)] attach independent Bernoulli spins to the other $\omega ^b$-open
clusters in such a way that the odds of cluster $C$ receiving spin
$1$ are $\aalpha e^{h|C|}$ to $1-\aalpha$.
\end{longlist}
We may assign spins $\s^b$ to the open clusters of the $\omega ^b$ in such
a way that: $\s^b$ has law $\pnh^b$, and $\s^0_y=\s^1_y$ for $y
\notin
C_x(\omega ^1)$. Write $\kappa^{01}_{n,h}$ for the joint law of the
ensuing pairs $(\omega ^0,\s^0)$, $(\omega ^1,\s^1)$.

When $h\le0$, let $\kappa^{01}_{n,h}$ be the coupling as above, with
the differences that: $\omega ^0\ge\omega ^1$, and $\s^0_y=\s^1_y$
for $y
\notin C_x(\omega ^0)$.
\end{pf}

We seek next a substitute for Lemma \ref{lemma2} in the current
setting. Let $J_{k,m,n}(x)$ be the conditional influence of vertex $x$
on the event $H_{k,m}$, with reference measure $\pnh$ on $\La_n$.

Let $(\omega ^0,\s^0,\omega ^1,\s^1)$ be sampled according to the measure
$\knh^{01}$ of Lemma \ref{coupling}. Define random clusters $C_x^H,
C_x^V\subseteq\ZZ^2$ as follows,
\begin{eqnarray*}
C_x^H(\omega ^0,\s^0,\omega ^1,\s^1)&:=&\{z\in\ZZ^2\dvtx\exists y \in
C_x(\omega ^0),\ y\lrab z \mbox{ in }\s^1\},\\
C_x^V(\omega ^0,\s^0,\omega ^1,\s^1)&:=&\{z\in\ZZ^2\dvtx\exists y \in
C_x(\omega ^1),\ y\lraw z \mbox{ in } \s^0\}.
\end{eqnarray*}
Notice that, if $h\ge0$ (\resp, $h\le0$), $C_x^H$ (\resp, $C_x^V$)
is the spin-$1$ cluster (\resp, spin-$0$ $*$-cluster) at $x$ under
$\s^1$ (\resp, $\s^0$). It may be checked as before that:
%
\begin{eqnarray}
J_{k,m,n}(x)&\le&\kappa^{01}_{n,h}(C_x^H \mbox{ contains a
horizontal crossing of } B_{k,m}),\label{ihp30}\\
J_{k,m,n}(x)&\le&\kappa^{01}_{n,h}(C_x^V \mbox{ contains a
vertical $*$-crossing of } B_{k,m}).\label{ihp31}
\end{eqnarray}
The notation $C_x^H$, $C_x^V$ is introduced in order to treat the
cases $h>0$ and $h<0$ simultaneously.

\begin{lem}\label{lemma3}
Let $R$ be as in Theorem \ref{crcmcr}.
\begin{longlist}[(ii)]
\item[(i)] If $\ttheta^1(p,q,\aalpha,\hc)=0$, and $\psh$ is subcritical for
$h\in[\hc-R,\hc]$, there exists $\nu_{k,1}$ satisfying $\nu
_{k,1}\to0$ as $k\to\oo$
such that
\[
\operatorname{lim\ sup}\limits_{n\to\oo}\sup_{h\in[\hc-R,\hc]} \,\sup_{x\in\La_n}
J_{k,m,n}(x)\le\nu_{k,1}.
\]

\item [(ii)] If $\ttheta^0(p,q,\aalpha,\hc)=0$, and $\psh$ is subcritical for
$h\in[\hc,\hc+R]$, there exists $\nu_{m,0}$ satisfying $\nu
_{m,0}\to0$ as $m\to\oo$
such that
\[
\operatorname{lim\ sup}\limits_{n\to\oo}\sup_{h\in[\hc,\hc+R]}\, \sup_{x\in\La_n}
J_{k,m,n}(x)\le\nu_{m,0}.
\]
\end{longlist}
\end{lem}

\begin{pf}
We prove part (i) only, the proof of (ii) being similar. If
$[\hc-R,\hc]\subseteq[0,\oo)$, let $\phi=\phi_{\hc}$; if
$[\hc-R,\hc]\subseteq(-\oo,0]$, let $\phi=\phi_{\hc-R}$. By
Proposition \ref{phi_properties}, and the assumptions of (i),
\begin{longlist}
\item[(a)] $\phi_{n,h} \lest\phi$ for $n\ge1$ and $h\in[\hc-R,\hc]$,
\item[(b)] $\phi$ is subcritical,
\item[(c)] $\pi_{\hc}$ is subcritical, and $\pi_{n,h}\lest\pi_{n,\hc}$ for
$h\in[\hc-R,\hc]$.
\end{longlist}
By Lemma \ref{crcm_finite_energy}, there exists $L>0$ such that
\begin{equation}
\pi_{n,h}(\s_x=1)\pi_{n,h}(\s_x=0) \ge L
\label{k7}
\end{equation}
for all $n \ge1$, $x\in\La_n$, and $h\in[\hc-R,\hc+R]$.
Let
\[
A_x(\omega ) = \sup\{r \ge0\dvtx x \lra x+\pd\La_r\}
\]
denote the radius $\operatorname{rad}(C_x)$ of the edge cluster $C_x = C_x(\omega)$ at $x$, and note that
$\phi(A_x<\oo)=1$.

Let $r\ge\max\{k,m\}$ and $x\in\La_r$.
By \eqref{ihp30} and the positive association of $\pnh^1$,
and as in \eqref{mel41},
\begin{eqnarray*}
J_{k,m,n}(x) &\le&\kappa^{01}_{n,h} \bigl(\rad(C_x^H) \ge k/2\bigr) \\
&\le&\sum_{a=0}^{\oo} \phi_{n,h}^0(A_x=a)\aalpha_{n,h}^1(x,a,k/2)\\
&\le&\frac1L \sum_{a=0}^{\oo} \phi_{n,h}(A_x=a)\aalpha_{n,h}(x,a,k/2),
\end{eqnarray*}
where
\[
\aalpha_{n,h}^\xi(x,a,b) = \pi_{n,h}^\xi(x+\La_a\lrab x+\pd\La_{b}\vert\s_y= 1 \mbox{ for }y\in
x+\La_a).
\]
Since $\aalpha_{n,h}(x,a,b)$ is nondecreasing in $a$, and furthermore $\phi
_{n,h} \lest\phi$ and
$\phi$ is translation-invariant,
\begin{equation}\label{k8-}
\sup_{x\in\La_r} \jkmn(x)
\le\frac1L \sum_{a=0}^\oo\phi(A_0=a) \sup_{x\in\La_r}\{\aalpha_{n,h}(x,a,k/2)\}.
\end{equation}
By \eqref{k7} and the fact that $\pi_{n,h} \lest\pi_{n,\hc}$,
\begin{equation}
\aalpha_{n,h}(x,a,k/2) \le
\min\biggl\{1, \frac1{L^{|\La_r|}} \pi_{n,\hc}(x+\La_a \lrao1
x+\pd\La_{k/2})\biggr\}.
\label{nar1}
\end{equation}

Suppose now that $x\in\La_n\sm\La_r$. Then
\begin{eqnarray*}
J_{k,m,n}(x) &\le&\kappa^{01}_{n,h} (C_x^H\cap B_{k,m}\not=\es)\\
&\le&\sum_{a=0}^\oo\phi_{n,h}^0(A_x=a) \bbb_{n,h}^1(x,a)\\
&\le&\frac1L \sum_{a=0}^\oo\phi_{n,h}(A_x=a)\bbb_{n,h}(x,a),
\end{eqnarray*}
where
\[
\bbb_{n,h}^\xi(x,a) =\pi_{n,h}^\xi
(x+\La_a\lrab B_{k,m} \vert\s_y= 1 \mbox{ for }y\in x+\La_a)
\]
is a nondecreasing function of $a$. Since $\phi_{n,h} \lest\phi$,
and $\phi$ is translation-invariant,
\[
\jkmn(x) \le \frac1L \sum_{a=0}^\oo\phi(A_0=a)\bbb_{n,h}(x,a).
\]
As above,
\begin{eqnarray*}
\bbb_{n,h}(x,a) &\le&\frac1{L^{|\La_a|}} \pi_{n,h}(x+\La_a \lrao1\bkm)\\
&\le&\frac1{L^{|\La_a|}} \pi_{n,h}(\bkm\lrao1 \pd\La_{r-a})\qq\mbox{if } a \le r,
\end{eqnarray*}
whence
\begin{equation}\label{k8}
\qq\jkmn(x) \le\frac1L \sum_{a=0}^\oo\phi(A_0=a)
\min\biggl\{1, \frac1{L^{|\La_a|}} \pi_{n,\hc}(\bkm\lrao1 \pd\La_{r-a})\biggr\},
\end{equation}
where the minimum is interpreted as $1$ when $a>r$.

We add \eqref{k8-}--\eqref{nar1} and \eqref{k8}, and take the limit
$n\to\oo$, to obtain
by the bounded convergence theorem that
\begin{eqnarray*}
&&\operatorname{lim\ sup}\limits_{n\to\oo} \sup_{x\in\La_n} \jkmn(x)\\
&&\qquad \le\frac1L \Biggl[ \sum_{a=0}^\oo
\phi(A_0=a) \min\biggl\{1, \frac1{L^{|\La_a|}} \pi_{\hc}(x+\La_a\lrao1 \pd\La_{k/2})\biggr\}\\
&&\qquad\hphantom{\le\frac1L \Biggl[}{}+\sum_{a=0}^\oo\phi(A_0=a)\min\biggl\{1, \frac1{L^{|\La_a|}} \pi_{\hc}(\bkm\lrao1 \pd\La
_{r-a})\biggr\}\Biggr].
\end{eqnarray*}

We now send $r\to\oo$. Since $\ttheta^1(p,q,\aalpha,\hc)=0$ by assumption,
the last summand tends
to $0$. By the bounded convergence theorem,
\begin{equation}\label{k9}
\limsup_{n\to\oo} \sup_{x\in\La_n} J_{k,m,n}(x) \le
\nu_{k,1},
\end{equation}
where
\[
\nu_{k,1} = \frac1L
\sum_{a=0}^\oo\phi(A_0=a)
\min\biggl\{1, \frac1{L^{|\La_a|}} \pi_{\hc}(x+\La_a \lrao1 \pd
\La_{k/2})\biggr\}.
\]
By the bounded convergence theorem again, $\nu_{k,1} \to0$ as $k \to
\oo$.
Since \eqref{k8-}--\eqref{nar1} and \eqref{k8} are
uniform in $h \in[\hc-R,\hc]$, one may include the supremum over $h$
in \eqref{k9}, as required for the lemma.
\end{pf}

\begin{pf*}{Proof of Theorem \ref{crcmcr}}
Let $f_n(h)=\pi_{n,h}(H_{k,m})$. By \eqref{vanc11} and Lemma \ref
{crcm_finite_energy},
%
\begin{equation}\label{k1}
\frac1{f_n(h)[1-f_n(h)]} \frac d{dh} f_n(h)\ge cL\log\biggl[\frac{1}{2\max_x J_{k,m,n}(x)}\biggr],
\end{equation}
with $L$ as in the proof of Lemma \ref{lemma3}.
Let
\[
\xi_{n,k,1}=\sup_{h\in[\hc-R,\hc]}\sup_{x\in\La_n} 2J_{k,m,n}(x),\qq\xi_{n,m,0}=\sup_{h\in[\hc,\hc+R]} \sup_{x\in\La_n} 2 J_{k,m,n}(x).
\]
By \eqref{k1},
\[
\log\frac{f_n(h)}{1-f_n(h)}\bigg|_{h_1}^{h_2}
\ge(\hc-h_1)cL\log(\xi^{-1}_{n,k,1}) + (h_2-\hc)cL\log(\xi^{-1}_{n,m,0}),
\]
whence
\[
f_n(h_1)[1-f_n(h_2)] \le\xi_{n,k,1}^{cL(\hc-h_1)}\xi_{n,m,0}^{cL(h_2-\hc)}.
\]
Take the limit as $n\to\oo$ and use Lemma \ref{lemma3}.
\end{pf*}

\begin{pf*}{Proof of Proposition \ref{phi_properties}}
A strictly positive measure $\mu$ on $\Om=\{0,1\}^E$ is monotone if
and only if: for all $\omega \in\Om$ with $\omega (e)=\omega
(f)=0$, $e\ne f$,
\begin{equation}\label{fkg}
 \mu(\omega ^{e,f})\mu(\omega )\ge\mu(\omega
^e)\mu(\omega ^f),
\end{equation}
see, for example, \cite{G-RC}, Theorem 2.19. Given two strictly positive
measures $\mu_1$ and $\mu_2$, at least one of which is monotone, it is
sufficient for $\mu_1\lest\mu_2$ that:
\begin{equation}\label{holley}
\frac{\mu_1(\omega ^e)}{\mu_1(\omega )} \le\frac{\mu_2(\omega
^e)}{\mu_2(\omega )},\qq\omega \in\Om,\, e\in E.
\end{equation}
This is proved in
\cite{G-RC}, Theorem 2.6. Condition \eqref{holley} is nontrivial
only when $\omega (e)=0$.

We shall prove (i) by checking that $\psh$ satisfies \eqref{fkg}.
Write $\sC(\omega )$ for the set of open clusters under $\omega $,
and let
$f_h(k)=\aalpha e^{hk}+1-\aalpha$. Substituting \eqref{crc4} into \eqref{fkg},
we must check
\begin{eqnarray}\label{fkg_phi}
&&\fpq(\omega ^{e,f})\fpq(\omega ) \prod_{C\in\sC(\omega ^{e,f})} f_h(|C|)
\prod_{C\in\sC(\omega )} f_h(|C|)\nonumber\\[-8pt]\\[-8pt]
&&\qq\ge\fpq(\omega ^{e})\fpq(\omega ^{f})\prod_{C\in\sC(\omega ^{e})} f_h(|C|) \prod_{C\in\sC(\omega
^{f})} f_h(|C|).\nonumber
\end{eqnarray}
On using the monotonicity of $\fpq$, and on canceling the factors
$f_h(|C|)$ for $C\in\sC(\omega )\cap\sC(\omega ^{e,f})$, we
arrive at the
following three cases.
\begin{longlist}[(iii)]
\item[(i)] There are clusters $C_1,C_2\in\sC(\omega )$, such that
$C_1\cup
C_2\in\sC(\omega ^e)=\sC(\omega ^f)$. It suffices that
\[
q f_h(a) f_h(b) \ge f_h(a+b),\qq a=|C_1|,\, b=|C_2|,
\]
and this is easily checked for $a,b\ge0$ since $q\aalpha,q(1-\aalpha)\ge1$.

\item[(ii)]  There are clusters $C_1,C_2,C_3\in\sC(\omega )$, such that
$C_1\cup
C_2\in\sC(\omega ^e)$ and $C_2\cup C_3\in\sC(\omega ^f)$. It
suffices that
\[
f_h(a+b+c) f_h(b) \ge f_h(a+b) f_h(b+c),\qq a=|C_1|,\, b=|C_2|,\,
c=|C_3|,
\]
and this is immediate.

\item[(iii)]  There are clusters $C_1,C_2,C_3,C_4\in\sC(\omega )$ such that
$C_1\cup C_2\in\sC(\omega ^e)$ and $C_3\cup C_4\in\sC(\omega
^f)$. In this
case, inequality \eqref{fkg_phi} simplifies to a triviality.
\end{longlist}

It may be checked similarly that the marginal measure of
$\kappa_h(\,\cdot\vert\s_x=b)$ on $\Om$ is monotone if either $h\ge0$,
$b=1$ or $h\le0$, $b=0$. One uses the expression
\[
\kappa_h(\omega \vert\s_x=b)\propto\phi_{p,q}(\omega )
e^{hb|C_x(\omega )|}
\prod_{C\in\sC(\omega )\sm\{C_x(\omega )\}} f_h(|C|),\qq\omega
\in\Om.
\]
Parts (ii) and (iii) then follow by checking \eqref{holley} with
appropriate $\mu_i$. Part (iv) follows from part (iii) by taking the
limit as $|h|\to\oo$. Many of the required calculations are rather
similar to part (i), and we omit further details.
\end{pf*}

\begin{pf*}{Proof of Theorem \ref{crcfk}}
We identify the spin-vector $\s\in\Si$ with the set $A=\{v\in V\dvtx
\s_v=1\}$. In order that $\pi=\ppqa$ be monotone, it is necessary and
sufficient [see inequality \eqref{fkg}] that
\begin{equation}\label{crc2}
\pi(A^{xy})\pi(A)\geq\pi(A^x)\pi(A^y),\qq A\subseteq V,\ x,y \in V\sm A,\ x\ne y.
\end{equation}
Let $A\subseteq V$, $x,y\in V\sm A$, $x\ne y$. Let $a$ be the
number of edges of the form $\langle x,z\ra$ with $z \in A$, let $b$
be the number of edges of the form $\langle x,z\ra$ with $z \notin A$
and $z \ne x,y$, and let $e$ be the number of edges joining $x$ and
$y$.

We write $A^x = A \cup\{x\}$, etc. By \eqref{crc1} with $h=0$,
\[
\frac{\pi(A^{x})}{\pi(A)} = (1-p)^{b+e-a}
\frac{Z_{A^x,q\aalpha}Z_{\ol{A^x},q(1-\aalpha)}}{Z_{A,q}Z_{\ol{A},q(1-\aalpha)}} =
\frac{\aalpha}{1-\aalpha} \cdot\frac{\phi_{\ol A,q(1-\aalpha)}(I_x)}
{\phi_{A^x,q\aalpha}(I_x)},
\]
where $I_x$ is the event that $x$ is isolated, and $\phi_{A,q}$ is the
\rc\ measure on the subgraph induced by vertices of $A$ with
edge-parameter $p$ and cluster-weight~$q$. Similarly,
\[
\frac{\pi(A^{xy})}{\pi(A^y)} =\frac{\aalpha}{1-\aalpha}
\cdot\frac{\phi_{\ol{A^y},q(1-\aalpha)}(I_x)} {\phi_{A^{xy},q\aalpha}(I_x)}.
\]
The ratio of the left to the right-hand sides of \eqref{crc2} is
\begin{equation}\label{crc3}
\frac{\phi_{A^x}(I_x)}{\phi_{A^{xy}}(I_x)} \cdot
\frac{\phi_{\ol{A^y}}(I_x)}{\phi_{\ol A}(I_x)} =
\frac{\phi_{A^{xy},q\aalpha}(I_x\vert I_y)}{\phi_{A^{xy},q\aalpha}(I_x)}
\cdot\frac{\phi_{{\ol A},q(1-\aalpha)}(I_x\vert I_y)}{\phi_{{\ol A},q(1-\aalpha)}(I_x)}.
\end{equation}
Inequality \eqref{crc2} holds by the positive association of
\rc\ measures with cluster-weights at least 1.

That the conditions are necessary for monotonicity follows by an
example. Suppose $0<q\aalpha<1$ and $q(1-\aalpha)\ge1$. Let $G$ be a cycle of
length four, with vertices (in order, going around the cycle)
$u,x,v,y$. Take $A=\{u,v\}$ above, so that $e=0$. The final ratio in
\eqref{crc3} equals 1, and the penultimate is strictly less than 1.
\end{pf*}

\begin{pf*}{Proof of Lemma \ref{crcm_finite_energy}}
By Proposition \ref{phi_properties}(iv) and inequality \eqref{mel29},
\[
\psh(I_x)\ge\phi_{p,Q}(I_x)\ge(1-p)^\De,
\]
where $I_x$ is the event that $x$ is isolated. Conditional on
$I_x$, the spin of $x$ under the coupling $\kh$ has the Bernoulli
distribution with parameter $\aalpha e^h/[\aalpha e^h+1-\aalpha]$.
\end{pf*}

\section*{Acknowledgments} The second author acknowledges the
hospitality of the Mathematics Department at the University of British
Columbia, the Institut Henri Poincar\'e--Centre Emile Borel,
Paris, and the Section de\break Math\'ematiques at the University of Geneva,
where this work was largely done. His stay in Geneva
was supported
by the Swiss National Science Foundation.
Rob van den Berg proposed the
elimination of equation \eqref{mel28} from the proofs.


\printaddresses

\end{document}